\documentclass{amsart}

\usepackage{amssymb}
\usepackage{amsmath}
\usepackage{amsthm}
\usepackage{amsbsy}
\usepackage{bm}
\usepackage{hyperref}
\date{\today}
\usepackage{cite}
\usepackage{array}

\theoremstyle{theorem}
    \newtheorem{theorem}{Theorem}
    \newtheorem{lemma}[theorem]{Lemma}

\theoremstyle{definition} 

    \newtheorem{result}[theorem]{Result}
    \newtheorem{remark}[theorem]{Remark}
    \newtheorem{example}[theorem]{Example}
    \newtheorem{exercise}[theorem]{Exercise}

\def\suchthat{\; : \;}

\def\l{\left}
\def\r{\right}
\def\<{\langle}
\def\>{\rangle}

\newcommand{\E}{\mbox{\bf E}}

\def\P{{\bf P}}

\newcommand\Tr{{\mbox{Tr}}}

\newcommand\mnote[1]{} 
\newcommand\be{\begin{equation*}}

\newcommand\ee{\end{equation*}}

\newcommand\ben{\begin{equation}}
\newcommand\een{\end{equation}}
\newcommand\bes{\begin{eqnarray*}}
\newcommand\ees{\end{eqnarray*}}

\newcommand\bex{\begin{exercise}}
\newcommand\eex{\end{exercise}}
\newcommand\beg{\begin{example}}
\newcommand\eeg{\end{example}}
\newcommand\benu{\begin{enumerate}}
\newcommand\eenu{\end{enumerate}}
\newcommand\beit{\begin{itemize}}
\newcommand\eeit{\end{itemize}}
\newcommand\berk{\begin{remark}}
\newcommand\eerk{\end{remark}}
\newcommand\bdefn{\begin{defintion}}
\newcommand\edefn{\end{definition}}
\newcommand\bthm{\begin{theorem}}
\newcommand\ethm{\end{theorem}}
\newcommand\bprf{\begin{proof}}
\newcommand\eprf{\end{proof}}
\newcommand\blem{\begin{lemma}}
\newcommand\elem{\end{lemma}}

\newcommand{\var}{\mbox{\rm Var}}

\newcommand{\Cov}{\mbox{\rm Cov}}
\newcommand{\sm}{{\raise0.3ex\hbox{$\scriptstyle \setminus$}}}

\def\l{\left}
\def\r{\right}

\def\CHI{\mathchoice%
{\raise2pt\hbox{$\chi$}}%
{\raise2pt\hbox{$\chi$}}%
{\raise1.3pt\hbox{$\scriptstyle\chi$}}%
{\raise0.8pt\hbox{$\scriptscriptstyle\chi$}}}
\def\smalloplus{\raise1pt\hbox{$\,\scriptstyle \oplus\;$}}

\title{Fluctuations of eigenvalues of  patterned random matrices}

\author{Kartick Adhikari}
\address{Department of Mathematics\\
        Indian Institute of Science\\
        Bangalore 560012, India}

\email{kartickmath [at] math.iisc.ernet.in}

\author{Koushik Saha}
\address{Department of Mathematics\\
        Indian Institute of Technology Bombay\\
         Powai, Mumbai  400076, India}

\email{ksaha [at] math.iitb.ac.in}

\date{\today}
\thanks{This work is partially supported by Inspire research grant of Koushik Saha and by Centre for Advanced Studies, IISc, Bangalore.}
\begin{document}
\maketitle
\begin{abstract}
In this article we  study  the fluctuation of linear statistics of eigenvalues of circulant, symmetric circulant, reverse circulant and Hankel matrices. We show that the linear spectral statistics of these matrices converge to the Gaussian distribution in total variation norm  when the matrices are constructed using i.i.d.  normal random variables. We also calculate the limiting variance of the linear spectral statistics for circulant, symmetric circulant and reverse circulant matrices.
\end{abstract}
\section{Introduction and main results}
Let $A_n$ be an $n\times n$ matrix with real or complex entries.  A linear statistics of
eigenvalues  $\lambda_1,\lambda_2,\ldots, \lambda_n$  of $A_n$ is a
function of the form
$$
\frac{1}{n}\sum_{k=1}^{n}f(\lambda_k),
$$
where $f$ is  some fixed function. 
The fluctuations of eigenvalues was first considered by Arharov \cite{arharov} for sample covariance matrices.  In 1982, Jonsson \cite{jonsson} proved the Central limit theorem
(CLT) of linear eigenvalue statistics  for  Wishart matrix and he used the method of moments to establish the result. In 1975 Girko   considered the CLT for the traces of resolvent of the Wigner and the sample covariance matrices using the Stieltjes transform and the martingale techniques (for results and references, see \cite{girko}).
In last two decades the fluctuations of linear statistics of eigenvalues of different type of random matrices have been  studied extensively. For recent  fluctuation  results on Wigner matrices and sample covariance matrices, we refer to see  \cite{bai2004clt}, \cite{johansson1998}, \cite{lytova2009central}, \cite{shcherbina2011central}, \cite{soshnikov1998tracecentral}  and the references therein. For results on  band and sparse type  random matrices, see \cite{anderson2006clt},  \cite{jana2014}, \cite{li2013central},  \cite{shcherbina2015}. For  fluctuation of eigenvalues of Toeplitz and band Toeplitz matrices, see  \cite{chatterjee} and \cite{liu2012}.

In this article we  study  the fluctuation of linear statistics of eigenvalues of some pattered matrices, namely, circulant, symmetric circulant, reverse circulant and Hankel matrices. All these matrices are well studied in mathematics and statistics literature. Circulant matrices play a crucial role in the study of large dimensional Toeplitz matrices with non-random input. See, for example,  \cite{gray2006} and \cite{grenander:szego}. The eigenvalues of the circulant matrices also arise crucially in time series analysis (see \cite{fanyao}, \cite{pollock}). The block version of generalized circulant matrices also arise  in areas such as multi-level supersaturated design of experiment \cite{georgiou} and spectra of De Bruijn graphs \cite{strok}.  For more results and application of circulant matrices, see also \cite{davis}. For recent progress on random circulant, reverse circulant matrices, we refer to  \cite{bhs:lsd:heavy}, \cite{bhs:poisson}, \cite{bhs:spectralnorm}, \cite{bhs:lsd}, \cite{bhs:spectralnorm:heavy}, \cite{massey2007}. Hankel matrix is closely related to reverse circulant matrix and Toeplitz matrix. For recent advancement on random Hankel matrix, we refer to \cite{bb2011}, \cite{bryc},\cite{liu2009limit}.  
Symmetric circulant, reverse circulant matrices also have deep connection with free probability theory. Limiting spectral distribution of these matrices are related to different notions of independence - classical independence and half independence (see \cite{BHSannals}, \cite{bhs:half}). However, there is no result in literature  on fluctuations of eigenvalues of circulant, symmetric circulant and  reverse circulant matrices, to the best of our knowledge. In  \cite{liu2012}, fluctuation of eigenvalues of random Hankel is considered. They established the CLT for linear statistics of eigenvalues of  Hankel matrix and  band Hankel matrix. 

The fluctuation problems, we are interested to consider for these patterned matrices, are mainly inspired by the following result on total variation norm convergence of linear spectral statistics of  Toeplitz matrices in \cite{chatterjee}.
\begin{result}[Theorem 4.5 in \cite{chatterjee}]\label{thm:sourav}
Consider the Gaussian Toeplitz matrices $T_n=(a_{i-j})_{i,j=0}^n$, where
 $a_j=a_{-j}$ and $\{a_j\}_{j=0}^{\infty}$ is a sequence of
independent standard Gaussian random variables. Let $p_n$ be a sequence of positive integers such that $p_n=o(\log n /\log \log n)$. Let $A_n = \frac{T_n}{\sqrt{n}}$, then as $n \to \infty$,
$$ \frac{tr(A_n^{p_n})-E[Tr(A_n^{p_n})]}{\sqrt{\var(Tr(A_n^{p_n}))}}\;\;\mbox{ converges in total variation  to $N(0, 1),$} $$
where $\var[X]$ denotes the variance of a random variable $X$.
The CLT also holds for $Tr(f(A_n))$ when $f$ is a fixed nonzero
polynomial with non negative coefficients.
\end{result}

Here we show that the above result holds for circulant matrix, symmetric circulant matrix, reverse circulant matrix and Hankel matrix. We also compute the limiting variance of the linear statistics of eigenvalues of circulant, symmetric circulant and reverse circulant matrices. The limiting variance of the linear statistics of eigenvalues of Hankel matrix is calculated in \cite{liu2012}. Before stating our main results we describe the structure of  these matrices.

A sequence is said to be an {\it input sequence} if the matrices are constructed from the given sequence. We consider the input sequence of the form $\{x_n: n\ge 0\}$. Circulant, symmetric circulant, reverse circulant and Hankel matrices are constructed from this given input sequence. 

\vspace{.2cm}
\noindent{\bf Circulant matrix :} The circulant matrix is defined as 
$$
C_n=\left(\begin{array}{cccccc}
x_0 & x_1 & x_2 & \cdots & x_{n-1} & x_{n-1} \\
x_{n-1} & x_0 & x_1 & \cdots & x_{n-3} & x_{n-2}\\
x_{n-2} & x_{n-1} & x_0 & \cdots & x_{n-4} & x_{n-3}\\
\vdots & \vdots & {\vdots} & \ddots & {\vdots} & \vdots \\
x_1 & x_2 & x_3 & \cdots & x_{n-1} & x_0
\end{array}\right).
$$
For $j=1,2,\ldots, n-1$, its $(j+1)$-th row is obtained by giving its $j$-th row a right circular shift by one positions and the (i,\;j)-th element of the matrix is $x_{(j-i) \mbox{ \tiny{mod} } n}$.

\vspace{.2cm}
\noindent{\bf Symmetric circulant matrix :} The symmetric circulant matrix is defined by
$$
SC_n=\left(\begin{array}{cccccc}
x_0 & x_1 & x_2 & \cdots & x_{2} & x_1 \\
x_1 & x_0 & x_1 & \cdots & x_{3} & x_{2}\\
x_{2} & x_1 & x_0 & \cdots & x_{4} & x_{3}\\
\vdots & \vdots & {\vdots} & \ddots & {\vdots} & \vdots \\
x_1 & x_2 & x_3 & \cdots & x_1 & x_0
\end{array}\right).
$$
For $j=1,2,\ldots, n-1$, its $(j+1)$-th row is obtained by giving its $j$-th row a right circular shift by one positions and the (i,\;j)-th element of the matrix is $x_{\frac{n}{2}-|\frac{n}{2}-|i-j||}$ Also note that symmetric circulant matrix is a Toeplitz matrix with the restriction that $x_{n-j}=x_j$.

\vspace{.2cm}
\noindent{\bf Reverse circulant matrix :} The reverse circulant matrix is defined as 
$$
RC_n=\left(\begin{array}{cccccc}
x_1 & x_2 & x_3 & \cdots & x_{n-1} & x_n \\
x_2 & x_3 & x_4 & \cdots & x_{n} & x_{1}\\
x_3 & x_4 & x_5 & \cdots & x_{1} & x_{2} \\
\vdots & \vdots & {\vdots} & \ddots & {\vdots} & \vdots \\
x_n & x_1 & x_2 & \cdots & x_{n-2} & x_{n-1}
\end{array}\right).
$$
For $j=1,2,\ldots, n-1$, its $(j+1)$-th row is obtained by giving its $j$-th row a left circular shift by one positions. Note that the matrix is symmetric and the (i,\;j)-th element of the matrix is $x_{(i+j-1) \mbox{ \tiny{mod} } n}$.

\vspace{.2cm}
\noindent{\bf Hankel matrix :} The Hankel matrix is defined as 
$$
H_n=\left(\begin{array}{cccccc}
x_1 & x_2 & x_3 & \cdots & x_{n-1} & x_n \\
x_2 & x_3 & x_4 & \cdots & x_{n} & x_{n+1}\\
x_3 & x_4 & x_5 & \cdots & x_{n+1} & x_{n+2} \\
\vdots & \vdots & {\vdots} & \ddots & {\vdots} & \vdots \\
x_n & x_{n+1} & x_{n+2} & \cdots & x_{2n-2} & x_{2n-1}
\end{array}\right).
$$
Note that this matrix is also a symmetric matrix and the (i,\;j)-th element of the matrix is $x_{(i+j-1) }$.

A input sequence $\{X_n:n\geq 0\}$  is said to be a {\it Gaussian input sequence} if the elements of the sequence are i.i.d. standard  normal random variables. In this article, all our  matrices are constructed from Gaussian input sequence. We have the following results on the fluctuation of linear statistics of eigenvalues of random circulant, symmetric circulant, reverse circulant and Hankel matrices. 

\begin{theorem}\label{thm:main:cir}
Let  $A_n$ be  an $n\times n$ circulant matrix or symmetric circulant matrix with Gaussian input sequence. 
Suppose $\{p_n\}$ is a sequence of positive integers such that $p_n=o(\log n/\log\log n)$. 
Then, as $n\to \infty$,
$$
\frac{\Tr(A_n^{p_n})-\E(\Tr(A_n^{p_n}))}{\sqrt{\var(\Tr(A_n^{p_n}))}} \;\mbox{converges in total variation norm to } N(0,1).
$$
\end{theorem}

\begin{theorem}\label{thm:main}
Let   $A_n$ be an $n\times n$  reverse circulant matrix or  Hankel matrix with Gaussian input sequence. 
Suppose $\{p_n\}$ is a sequence of positive \textbf{even} integers such that $p_n=o(\log n/\log\log n)$. Then, as $n\to \infty$,
$$
\frac{\Tr(A_n^{p_n})-\E(\Tr(A_n^{p_n}))}{\sqrt{\var(\Tr(A_n^{p_n}))}} \;\mbox{converges in total variation norm to } N(0,1).
$$
\end{theorem}

\noindent In the next result, we compute the exact rate of variance of $\Tr(C_n^p)$ for a fixed positive integer $p$.  
\begin{theorem}\label{thm:limvarcirculant}
Fix a positive integer $p$. Then 
$$
\lim_{n\to \infty}\frac{\var[\Tr(C_n^p)]}{n^{p+1}}=p! \sum_{s=0}^{p-1}f_p(s),
$$ 
where $f_p$ is the density of the random variable $U_1+U_2+\cdots+U_p$ and $U_1,U_2,\ldots, U_p$ are i.i.d. uniform random variables on $[0,1]$. 
\end{theorem}
The distribution of $U_1+\cdots+U_p$ is known as Irwin-Hall distribution. Clearly, the random variable is supported on $[0,p]$. The density function $f_p$ of $(U_1+U_2+\cdots+U_p)$ is given by
\begin{align}\label{eqn:density}
f_p(x)=\frac{1}{(p-1)!}\sum_{k=0}^{\lfloor x \rfloor}(-1)^k\binom{p}{k}(x-k)^{p-1}, \;\;\mbox{ when $x\in [0,p]$},
\end{align}
where $\lfloor x \rfloor$ denotes the largest integer not exceeding $x$.

In the next result, we  compute the exact rate of variance for reverse circulant matrix when $p_n$ is a fixed even integer.  The limiting variance for linear statistics of eigenvalues of  Hankel matrix is calculated in Theorem 6.4 of \cite{liu2012}. 
\begin{theorem}\label{thm:reverseciculant}
Fix a positive integer $p$. Then 
$$
\lim_{n\to\infty}\frac{\var[\Tr(RC_n^{2p})]}{n^{2p+1}}=\sum_{k=2}^pc_kg(k)+2c_1,
$$
 where $c_k=\l(\binom{p}{p-k}^2(p-k)!\r)^2$ and $g(k)$ is given by
 $$
 g(k)=\frac{1}{(2k-1)!}\sum_{s=-(k-1)}^{k-1}\sum_{j=0}^{k+s-1}(-1)^j\binom{2k}{j}(k+s-j)^{2k-1}(2-{\bf 1}_{\{s=0\}})k!k!.
 $$
\end{theorem}
%
%
The next theorem provides  the exact rate of variance of $\Tr(SC_n^p)$ for a fixed positive integer $p$. 
\begin{theorem}\label{thm:scvariance}
Fix a  integer $p$.  
(i) Then for $p=2m+1$,
\begin{align*}
&\lim_{n\to \infty}\frac{\var[\Tr(SC_n^p)]}{n^{p+1}}\\
&=(2m+1)^2\binom{2m}{m}^2(m!)^2\frac{1}{2^{2m}}+\sum_{k=1}^{m}\frac{a_k}{2^{2(m-k)}}\sum_{\ell=0}^{2k+1}\binom{2k+1}{\ell}^2 \ell!(2k+1-\ell)! \ h_{2k+1}(\ell),
\end{align*}
where $a_k=(\binom{2m+1}{2m-2k}\binom{2m-2k}{m-k} (m-k)!)^2$ and 
\begin{align*}
h_{2k+1}(\ell)= \frac{1}{(2k)!}\sum_{s=-\lceil\frac{2k+1-\ell}{2}\rceil}^{\lfloor\frac{\ell}{2}\rfloor}\sum_{q=0}^{2s+2k+1-\ell}(-1)^q\binom{2k+1}{q}\l(\frac{2s+2k+1-\ell-q}{2}\r)^{2k}.
\end{align*}
(ii) For $p=2m$,
$$
\lim_{n\to \infty}\frac{\var[\Tr(SC_n^p)]}{n^{p+1}} =\sum_{k=2}^{m}\frac{b_k}{2^{2(m-k)}}\l(\sum_{\ell=0,\ell\neq k}^{2k}\binom{2k}{\ell}^2 \ell!(2k-\ell)! \  h_{2k}(\ell)+\binom{2k}{k}^2g(k)\r)+\frac{b_1}{2^{(2m-1)}},
$$
where $b_k=(\binom{2m}{2m-2k}\binom{2m-2k}{m-k}(m-k)!)^2$ and $g(k)$ as in the Theorem \ref{thm:reverseciculant}.
\end{theorem}



The spectral norm of the random matrices play a crucial role in the proofs of Theorems \ref{thm:main:cir} and \ref{thm:main}. In \cite{meckes}, Meckes showed that the spectral norm of random Toeplitz matrix is the order of $\sqrt{n\log n} $  and mentioned that the spectral norm of random Hankel matrix is also of the same order. He also pointed out that the methods of his paper can be used to treat  random Toeplitz matrix with extra restrictions, for example, random symmetric circulant matrix which is a Toeplitz matrix with the restriction that $x_{n-j}=x_j$.    Using his method in \cite{meckes}, we show that the spectral norm of random circulant, reverse circulant and symmetric circulant  matrices with Gaussian input sequence are  of the order of $\sqrt{n\log n }$. As a consequence of this we get the order of the spectral norm of random Hankel matrix with Gaussian input sequence. We denote the spectral norm of a matrix $A$ by $\|A\|$.

\begin{theorem}\label{thm:normcir}
Let $A_n$ be an  $n\times n$ matrix among circulant, symmetric circulant, reverse circulant and Hankel matrices with Gaussian input sequence. Then 
\begin{align*}
\limsup_{n \to \infty}\frac{\|A_n\|}{\sqrt{n\log n}}\le C \;\;\mbox{ a.s.},
\end{align*}
where $C$ is a positive constant.
\end{theorem}

The rest of our paper is organized as follows. In Section \ref{sec:proof}, we give the proofs of Theorem \ref{thm:main:cir} and Theorem \ref{thm:main}. In  Section \ref{sec:limvar}, we give the proofs of Theorem \ref{thm:limvarcirculant}, Theorem \ref{thm:reverseciculant} and Theorem \ref{thm:scvariance}. In the appendix, we  give a  proof of Theorem \ref{thm:normcir} and a few  norm related results of pattered matrices. 

\section{Proofs of Theorem \ref{thm:main:cir} and Theorem \ref{thm:main}}\label{sec:proof}
In this section we prove of Theorem \ref{thm:main:cir} and Theorem \ref{thm:main}. The key ingredient in both the proofs is Proposition 4.4 in \cite{chatterjee}. We state a version of the result which will be used in the proofs of our theorems.   Suppose we have a collection $A_n=(a_{ij})_{1\le i, j\le n}$ of jointly Gaussian random variables with mean zero and $n^2\times n^2$ covariance matrix $\Sigma_{A_n}$. 
\begin{result}[Proposition 4.4 in \cite{chatterjee}]\label{res:sourav}
Fix a positive integer $p$. Suppose $W_n=\Tr (A_n^p)$ and $\sigma^2=\var(W_n)$. Let $Z_n$ be a normal random variable with  same mean and variance as $W_n$. Then 
$$
d_{TV}(W_n,Z_n)\le \frac{2\sqrt 5 \|\Sigma_{A_n}\|^{3/2}ab}{\sigma^2 n},
$$
where $a=p(\E (\|A_n\|)^{4(p-1)})^{\frac{1}{4}}$, $b=p(p-1)(\E(\|A_n\|)^{4(p-2)})^{\frac{1}{4}}$ and $d_{TV}(W_n,Z_n)$ denotes the total variation distance between $W_n$ and $Z_n$.
\end{result}

The next two   lemmas give a lower bound for $\var[\Tr(C_n^p)]$, $\var[\Tr(SC_n^p)]$, $\var[\Tr(RC_n^{2p})]$ and $\var[\Tr(H_n^{2p})]$. 
\begin{lemma}\label{lem:lowvar:cir}
Let $A_n$ be an $n\times n$ circulant matrix or symmetric  circulant matrix with Gaussian input sequence. Then for large $n$,
$$
\var[\Tr(A_n^p)]\ge \frac{n^{p+1}}{(12p)^{p}}
$$
where $p\ge 2$, a fixed positive integer.
\end{lemma}
\begin{lemma}\label{lem:lowvar}
Let $A_n$ be   an $n\times n$ reverse circulant matrix or  Hankel matrix with Gaussian input sequence. Then for large $n$,
$$
\var[\Tr(A_n^{p})]\ge \frac{n^{p+1}}{(3p)^{p+1}}
$$
where $p\ge 2$, a fixed positive even integer.
\end{lemma}
The next lemma gives an upper bound on the spectral norm of the covariance matrix $\Sigma_{A_n}$ where $A_n$ is one of the  matrix  among circulant, reverse circulant and Hankel matrices.
\begin{lemma}\label{lem:cov}
Suppose $A_n$ is one of the matrix among random circulant, symmetric circulant,  reverse circulant and Hankel matrices with Gaussian input sequence. Then 
$$
\|\Sigma_{A_n}\|\le cn,
$$
where $c=1$ for $C_n,RC_n, H_n$ and $c=2$ for $SC_n$.\end{lemma}

\noindent Assuming Lemma \ref{lem:lowvar:cir} and Lemma \ref{lem:cov},  now we  prove Theorem \ref{thm:main:cir}. 

\begin{proof}[Proof of Theorem \ref{thm:main:cir}] 
Let $A_n$ be an $n\times n$ circulant matrix $C_n$ or symmetric  circulant matrix $SC_n$ with the Gaussian input sequence. Using Lemma \ref{lem:lowvar:cir} and Lemma \ref{lem:cov} in Result \ref{res:sourav}, we have 
\begin{align*}
d_{TV}(W_n,Z_n)\le \frac{2\sqrt 5(2n)^{3/2}(12p)^{p}p^3(\E\|A_n\|^{4p})^{1/2}}{n^{p+1}n},
\end{align*} 
where $W_n=\Tr (A_n^p)$ and $Z_n$ is a normal random variable with  same mean and variance as $W_n$.
Now from Theorem \ref{thm:normcir}, we have
\begin{align*}
d_{TV}(W_n,Z_n)\le \frac{C'(12C)^p p^{p+3} n^{3/2} (n\log n)^p}{n^{p+2}}=\frac{C' (12C)^p p^{p+3} (\log n)^p}{\sqrt n},
\end{align*}
where $C'$ is a constant does not depend on $p$ and $n$. Clearly, if $p=o(\log n/\log\log n)$, then 
$d_{TV}(W_n,Z_n)$ goes to zero as $n\to \infty$. Hence the result.
\end{proof}

\vspace{.2cm}
\noindent Proof of Theorem \ref{thm:main} is similar to the proof of Theorem \ref{thm:main:cir}. One has to use Lemma \ref{lem:lowvar} instead of Lemma \ref{lem:lowvar:cir} to prove Theorem \ref{thm:main}. Here we skip the details. But we note that the result in Lemma \ref{lem:lowvar} is true when 	 $p$ is an even integer and hence, the result in Theorem \ref{thm:main} holds when $p_n$ is a sequence of even integers. 
Now we give the proofs of the lemmas stated in this section.
\vspace{.2cm}
\begin{proof}[Proof of Lemma \ref{lem:lowvar:cir}]
\noindent{\bf Circulant matrix :} 
Let $\lambda_1,\lambda_2,\ldots,\lambda_{n}$ be the eigenvalues of the random circulant matrix $C_n$ with Gaussian input sequence. Then $\lambda_k$'s are given (see \cite{bhs:lsd}, Section 2.2) by
\begin{align*}
\lambda_k=\sum_{j=0}^{n-1}X_je^{i\omega_kj}=\sum_{j=0}^{n-1}X_j\cos(\omega_k j)+i\sum_{j=1}^{n}X_j\sin(\omega_k j),
\end{align*}
where $\omega_k=\frac{2\pi k}{n}$ and  $k=1,2,\ldots, n$. Now observe that $\Tr(C_n^p)$ is given by
\begin{equation}\label{trace:C_n}
\Tr(C_n^p)=n\sum_{A_p}X_{i_1}\cdots X_{i_p},
\end{equation}
where $A_p=\{(i_1,\ldots,i_p)\suchthat i_1+\cdots+i_p=0\;{(\mbox{mod $n$})},\; 0\le i_1,\ldots, i_p\le n-1\}.$  We show that all the terms in the above sum are positively correlated. Note, 
\begin{align*}
\Cov\l(\prod_{i=1}^n X_i^{\alpha_i},\prod_{i=1}^n X_i^{\beta_i}\r)=\prod_{i=1}^n \E[ X_i^{\alpha_i+\beta_i}]-\prod_{i=1}^n \E [X_i^{\alpha_i}]\E[X_i^{\beta_i}].
\end{align*}
Now, if $\alpha_i+\beta_i$ is odd, then $\E[ X_i^{\alpha_i+\beta_i}]= \E [X_i^{\alpha_i}]\E[X_i^{\beta_i}]=0$. If $\alpha_i$ and $\beta_i$ are both odd, then $\E [X_i^{\alpha_i+\beta_i}]\ge 0$ and $\E [X_i^{\alpha_i}]=\E[X_i^{\beta_i}]=0$. Finally, if $\alpha_i$ and $\beta_i$ are both even, then 
by H$\ddot{o}$lder's inequality
\begin{align*}
\E [X_i^{\alpha_i}]\E[X_i^{\beta_i}]\le \l(\E[X_i^{\alpha_i+\beta_i}]\r)^{\frac{\alpha_i}{\alpha_i+\beta_i}}\l(\E[X_i^{\alpha_i+\beta_i}]\r)^{\frac{\beta_i}{\alpha_i+\beta_i}}=\E[X_i^{\alpha_i+\beta_i}].
\end{align*}
Therefore, in all the cases, we have 
\begin{align*}
\E[ X_i^{\alpha_i+\beta_i}]\ge \E [X_i^{\alpha_i}]\E[X_i^{\beta_i}]. 
\end{align*}
Hence the terms in the sum  \eqref{trace:C_n} are positively correlated. Therefore 
\begin{align*}
\var(\Tr((C_n)^p))\ge n^2\var\l(\sum_{A_p'}X_{i_1}\cdots X_{i_p}\r),
\end{align*}
where $A_p'=\{(i_1,\ldots,i_p)\suchthat i_1+\cdots+i_p=n,\; 1\le i_1,\ldots, i_p\le n\}.$ Note that $|A_p'|$, the cardinality of $A_p'$,  is $\binom{n-1}{p-1}$ and $\var(X_i^{\alpha_i})\ge 1$ for $\alpha_i\ge 1$. Therefore 
\begin{align*}
\var(\Tr((C_n)^p))\ge n^2 |A_p'|=n^2\binom{n-1}{p-1}\ge n^2 \frac{(n-p)^{p-1}}{(p-1)!}.
\end{align*}
This completes the proof for circulant matrix.

\vspace{.2cm}
\noindent{\bf Symmetric Circulant matrix :} Let $SC_n$ be an $n\times n$ symmetric circulant matrix with Gaussian input sequence and $p$ be a fixed positive integer. Then 
\begin{align*}
\Tr(SC_n^p)=\sum_{i_1,\ldots, i_p=1}^na_{i_1i_2}a_{i_2i_3}\cdots a_{i_pi_1},
\end{align*}
where $a_{ij}=X_{\tiny\frac{n}{2}-|\tiny\frac{n}{2}-|i-j||}$. Note that the terms in the last equation are positively correlated and $a_{ij}=X_{|i-j|}$ when $|i-j|\le \tiny\frac{n}{2}$. Therefore, we have 
\begin{align*}
\var[\Tr(SC_n^p)]\ge \var\l[ \sum_{i_1,\ldots, i_p=1}^{n/2}X_{|i_1-i_2|}X_{|i_2-i_3|}\ldots X_{|i_p-i_1|}\r].
\end{align*}
Now following the argument given to prove equation (12)  in the proof of Theorem 4.5 of \cite{chatterjee}, we have
$$
\var[\Tr(SC_n^p)]\ge \frac{n^{p+1}}{(12p)^p}.
$$
This completes the proof.
\end{proof}

\begin{proof}[Proof of Lemma \ref{lem:lowvar}]
\noindent{\bf Reverse circulant matrix :} Consider random reverse circulant matrix $RC_n=(a_{ij})_{1\le i,j\le n}$, where $a_{ij}=X_{i+j-1\;\tiny\mbox{(mod $n$)}}$ and $X_1,\ldots, X_n$ are i.i.d. standard normal random variables. Then 
\begin{align*}
\Tr((RC_n)^p)=\sum_{i_1,\ldots,i_p=1}^{n}a_{i_1i_2}a_{i_2i_3}\cdots a_{i_pi_1}.
\end{align*}
As all the terms in the above sum are positively correlated, we have 
\begin{align}\label{eqn:var}
\var(\Tr((RC_n)^p))&=\var\l(\sum_{i_1,\ldots,i_p=1}^{n}a_{i_1i_2}a_{i_2i_3}\cdots a_{i_pi_1}\r)\nonumber
\\&\ge \var\l(\sum_{i_1,\ldots,i_p=1}^{n/3}X_{i_1+i_2-1}X_{i_2+i_3-1}\cdots X_{i_p+i_1-1}\r).
\end{align}
Let
\begin{align*}
D_{a_1,\ldots,a_p}=\{ (i_1,\ldots,i_p)\suchthat1\leq i_1\leq \frac{n}{3p}, i_k+i_{k+1}-1=a_k,k=1,\ldots,p\},
\end{align*}
where $i_{p+1}=i_1$ and $(a_1,\ldots,a_p)\in \mathcal A$ and 
$$
\mathcal A=\l\{(a_1,\ldots,a_p)\in \mathbb N^p\suchthat\frac{kn}{3p}+1< a_k\le \frac{(k+1)n}{3p},k=1,2,\ldots,p-1\r\}.
$$ 
Here we write  $\frac{n}{3p},\ldots, \frac{n}{3}$ as if integers, which do not effect the asymptotic. Observe that the set $D_{a_1,\ldots,a_p}$ are disjoint sets for different values of $(a_1,\ldots,a_p)$. Therefore from \eqref{eqn:var}, we have
\begin{align}\label{eqn:var:lower}
\var(\Tr((RC_n)^p))&\ge \sum_{\mathcal A}\var\l(\sum_{D_{a_1,\ldots,a_p}}X_{i_1+i_2-1}X_{i_2+i_3-1}\cdots X_{i_p+i_1-1}\r).
\end{align}
Now, for a fixed value of $a_p$, $|\mathcal A|=(\frac{n}{3p})^{p-1}$. Again, from the definition of $a_1,a_2,\ldots,a_p$, we have
$$i_p=\left\{
\begin{array}{ccc}
a_{p-1}-a_{p-2}+\cdots-a_2+a_1-i_1+1 &\mbox{if}& p \mbox{ is even,}\\ 
a_{p-1}-a_{p-2}+\cdots+a_2-a_1+i_1  &\mbox{if}& p \mbox{ is odd.}
\end{array}
\right.
$$
This implies that $a_p=i_p+i_1-1=a_{p-1}-a_{p-2}+\cdots-a_2+a_1$, when $p$ is even. Therefore, if $p$ is even then $a_p$ is determined by $a_1,a_2,\ldots,a_{p-1}$ and it does not depend on $i_1$. But, if $p$ is odd then $a_p$ depends on $a_1,a_2,\ldots,a_{p-1}$ and $i_1$.

So, if $p$ is even then  for a fixed choice of $a_1,\ldots,a_{p-1}$ the number of elements in $D_{a_1,\ldots,a_p}$ is same as the number of ways of choosing $i_1$. Hence 
$$
|D_{a_1,\ldots,a_p}|=\frac{n}{3p}, \mbox{ if } p \mbox{ is even.}
$$ 
 Thus, when $p$ is even, from \eqref{eqn:var:lower} we get
 $$\var(\Tr((RC_n)^p))\ge |\mathcal A||D_{a_1,\ldots,a_p}|^2 \var(X_{a_1}\cdots X_{a_p})\ge\frac{n^{p+1}}{(3p)^{p+1}}.$$
In the last inequality we have used the fact that $\var(X_1)=1$ and  $\var(X_1^2)=2$. This completes the proof for reverse circulant matrix.

\vspace{.2cm}
\noindent{\bf Hankel matrix :} The structure of Hankel matrix ($(i,j)$-th entry  is $X_{i+j-1}$) is very close to the structure of reverse circulant matrix ($(i,j)$-th entry  is $X_{i+j-1\;\tiny\mbox{(mod $n$)}}$). Following  the similar argument given for reverse circulant matrix,  one can show that, for $p$ even, 
$$
\var(\Tr((H_n)^p))\ge \frac{n^{p+1}}{(3p)^{p+1}}.
$$
Hence the result.
\end{proof}

\begin{proof}[Proof of Lemma \ref{lem:cov}]
The random matrices $C_n,RC_n$ and $H_n$ are constructed from a Gaussian input sequence. 
Therefore elements of $\Sigma_{C_n}$ are given by
$$
\sigma_{ij,kl}=\l\{\begin{array}{ll}
1 & \mbox{ if } j-i=l-k \;\;(\mbox{mod } n)\\ 0 & \mbox{otherwise,}
\end{array}\r.
$$
for $1\le i,j,k,l\le n$. Note that the entries of $\Sigma_{RC_n}$ are given by
$$
\sigma_{ij,kl}=\l\{\begin{array}{ll}
1 & \mbox{ if } i+j=k+l \;\;(\mbox{mod } n)\\ 0 & \mbox{otherwise,}
\end{array}\r.
$$
for $1\le i,j,k,l\le n$. Similarly the entries of $\Sigma_{H_n}$ are given by
$$
\sigma_{ij,kl}=\l\{\begin{array}{ll}
1 & \mbox{ if } i+j=k+l\\ 0 & \mbox{otherwise,}
\end{array}\r.
$$
for $1\le i,j,k,l\le n$. Observe that, in all the above cases the number of $1$ in a row of $\Sigma_{A_n}$ is at most $n$.
Therefore, by Gershgorin bound for the operator norm, 
\begin{align*}
\|\Sigma_{A_n}\|\le \max_{i,j}\sum_{k,l}|\sigma_{ij,kl}|\le n,
\end{align*}
where $A_n$ is one of the matrix among random circulant, reverse circulant and Hankel matrices with Gaussian input sequence.

For symmetric circulant matrix  with Gaussian input sequence,  elements of $\Sigma_{SC_n}$ are given by
$$
\sigma_{ij,kl}=\l\{\begin{array}{ll}
1 & \mbox{ if } |\frac{n}{2}-|i-j|| =|\frac{n}{2}-|k-l|| \\ 0 & \mbox{otherwise,}
\end{array}\r.
$$
Observe that the number of $1$ in a row of $\Sigma_{SC_n}$ is at most $2n$. Therefore, by Gershgorin bound for the operator norm, we have
\begin{align*}
\|\Sigma_{SC_n}\|\le \max_{i,j}\sum_{k,l}|\sigma_{ij,kl}|\le 2n.
\end{align*}
This completes the proof.
\end{proof}

\section{Proof of Theorems \ref{thm:limvarcirculant}, \ref{thm:reverseciculant} and \ref{thm:scvariance}}\label{sec:limvar}
In this section we give the proofs of Theorem \ref{thm:limvarcirculant}, Theorem \ref{thm:reverseciculant} and Theorem \ref{thm:scvariance}.
\subsection{Proof of theorem \ref{thm:limvarcirculant}}
We  introduce some notations before proving the theorem.
\begin{align*}
A_p&=\{(i_1,\ldots,i_p)\suchthat i_1+\cdots+i_p=0\;{(\mbox{mod $n$})},\; 0\le i_1,\ldots, i_p\le n-1\},\\
A_{p,s}&=\{(i_1,\ldots,i_p)\suchthat i_1+\cdots+i_p=sn,\; 0\le i_1,\ldots, i_p\le n-1\},\\
A_{p,s}'&=\{(i_1,\ldots,i_p)\suchthat i_1+\cdots+i_p=sn,\; 0\le i_1\neq i_2\neq \cdots\neq i_p\le n-1\}.
\end{align*}
The following lemma gives the cardinality of  $A_{p,s}$, which will be used in the proof of Theorem \ref{thm:limvarcirculant}.
\begin{lemma}\label{lem:cardinality}
Suppose $|A_{p,s}|$ denotes   the cardinality of $A_{p,s}$. Then 
$$
|A_{p,s}|=\sum_{k=0}^{s}(-1)^k\binom{p}{k}\binom{(s-k)n+p-1}{p-1}, \;\;\mbox{ for $s=0,1,2,\ldots,p-1$.}
$$ 
\end{lemma}

\noindent Assuming this lemma we proceed to prove  Theorem \ref{thm:limvarcirculant}.

\begin{proof}[Proof of theorem \ref{thm:limvarcirculant}] Recall that  we have $\Tr(C_n^p)=n\sum_{A_p}X_{i_1}\cdots X_{i_p}$, for fixed positive integer $p$.
Therefore we have 
\begin{align*}
\frac{1}{n^{p+1}}\var[\Tr(C_n^p)]=\frac{1}{n^{p-1}}\l[\E\l(\sum_{A_p}X_{i_1}\cdots X_{i_p}\r)^2-\l(\E\sum_{A_p}X_{i_1}\cdots X_{i_p}\r)^2\r]
\end{align*}
Note that if $p$ is odd then $\E\l(\sum_{A_p}X_{i_1}\cdots X_{i_p}\r)=0$. Again when $p$ is even then $\E\l(\sum_{A_p}X_{i_1}\cdots X_{i_p}\r)=O(n^{p/2-1})$. Therefore for any positive integer $p$, we have 
\begin{align*}
\lim_{n\to \infty}\frac{1}{n^{p/2-1/2}} \E\l(\sum_{A_p}X_{i_1}\cdots X_{i_p}\r)=0
\end{align*}
and, hence 
\begin{align*}
\lim_{n\to \infty}\frac{1}{n^{p+1}}\var[\Tr(C_n^p)]&=\lim_{n\to \infty}\frac{1}{n^{p-1}}\l[\E\l(\sum_{A_p}X_{i_1}\cdots X_{i_p}\r)^2\r]\\
&=\lim_{n\to \infty}\frac{1}{n^{p-1}}\sum_{A_p}\sum_{A_p}\E[X_{i_1}\cdots X_{i_p}X_{j_1}\cdots X_{j_p}]\\
&=\lim_{n\to \infty}\frac{1}{n^{p-1}}\sum_{s,t=0}^{p-1}\sum_{A_{p,s}}\sum_{A_{p,t}}\E[X_{i_1}\cdots X_{i_p}X_{j_1}\cdots X_{j_p}].
\end{align*}
Note that the number of solutions of $i_1+\cdots+i_p=sn$, with $0\le i_1,\ldots,i_p\le n-1$, is $O(n^{p-1})$. And if atleast two indices are equal then the number of solutions of the same equation is $O(n^{p-2})$. 

Again $\E[X_{i_1}\cdots X_{i_p}X_{j_1}\cdots X_{j_p}]$ will be non zero only when the random variables appear with even order. Observe that if $s\neq t$ then there is atleast one self matching in $\{i_1,\ldots,i_p\}$ and there exists a self matching with in $\{j_1,\ldots,j_p\}$, therefore in such cases the number of non zero terms is $O(n^{p-2})$. Therefore we have 
\begin{align}\label{eqn:1}
\lim_{n\to \infty}\frac{1}{n^{p+1}}\var[\Tr(C_n^p)]&=\lim_{n\to \infty}\frac{p!}{n^{p-1}}\sum_{s=0}^{p-1}\sum_{A_{p,s}'}\E[X_{i_1}^2\cdots X_{i_p}^2]=p!\sum_{s=0}^{p-1}\lim_{n\to \infty}\frac{|A_{p,s}'|}{n^{p-1}}.
\end{align}
The last equality follows from the fact that $\E[X_{i_1}^2\cdots X_{i_p}^2]=\E[X_{i_1}^2]\cdots \E[X_{i_p}^2]=1$, as $i_1,\ldots, i_p$ are distinct number and $X_0,X_1,\ldots,X_{n-1}$ are independent with mean zero and variance one. The factor $p!$ appeared because $\{j_1,j_2,\ldots, j_p\}$ can match with given vector $(i_1,i_2,\ldots,i_p)$ in $p!$ ways. Note that the number of cases  for which atleast one equality holds in the indices of $(i_1,\ldots,i_p)\in A_{p,s}$  is $O(n^{p-2})$. Therefore we have 
\begin{align}\label{eqn:2}
\lim_{n\to \infty}\frac{|A_{p,s}'|}{n^{p-1}}=\lim_{n\to \infty}\frac{|A_{p,s}|}{n^{p-1}}\;\;\mbox{ for $s=0,1,2,\ldots, p-1$.}
\end{align}
Since $p$ is a fixed positive integer, by Lemma \ref{lem:cardinality},  for $s=0,1,2,\ldots, p-1$, we get 
\begin{align}\label{eqn:3}
\lim_{n\to \infty}\frac{|A_{p,s}|}{n^{p-1}}=\frac{1}{(p-1)!}\sum_{k=0}^{s}(-1)^{k}\binom{p}{k}(s-k)^{p-1}=f_p(s),
\end{align}
where $f_p$ is the probability density function of Irwin-Hall distribution (as in \eqref{eqn:density}).  Therefore by \eqref{eqn:1}, \eqref{eqn:2} and \eqref{eqn:3} we have
\begin{align*}
\lim_{n\to \infty}\frac{1}{n^{p+1}}\var[\Tr(C_n^p)]&=p!\sum_{s=0}^{p-1}f_p(s).
\end{align*}
Hence the result.
\end{proof}

\noindent Now we   prove Lemma \ref{lem:cardinality}.

\begin{proof}[Proof of Lemma \ref{lem:cardinality}]
Observe that, for fixed positive integer $p$ and $0\leq s<p$,  $|A_{p,s}|$ is the coefficient of $x^{sn}$ in the expression $(1+x+x^2+\cdots+x^{n-1})^p$. We have
\begin{align*}
(1+x+x^2+\cdots+x^{n-1})^p&= (1-x^n)^p(1-x)^{-p}
\\&=\l(\sum_{k=0}^p(-1)^k\binom{p}{k}x^{kn}\r)\l(\sum_{m=0}^{\infty}\binom{p+m-1}{p-1}x^m\r).
\end{align*}
It is clear from the last equation that the coefficient of $x^{sn}$ is 
\begin{align*}
\sum_{k=0}^{s}(-1)^k\binom{p}{k}\binom{(s-k)n+p-1}{p-1},
\end{align*}
for $s=0,1,2,\ldots, p-1$. This completes the proof.
\end{proof}

\vspace{.2cm}
\subsection{Proof of Theorem \ref{thm:reverseciculant}} We introduce some notations which will be used in the proof of  Theorem \ref{thm:reverseciculant}.
\begin{align*}
B_{2p}&=\{(i_1,\ldots,i_{2p})\in \mathbb N^{2p}\suchthat \sum_{k=1}^{2p}(-1)^ki_k=0 \mbox{ (mod $n$) }, 1\le i_1,\ldots,i_{2p}\le n\},\\
B_{2p}'&=\{(i_1,\ldots,i_{2p})\in \mathbb N^{2p}\suchthat \sum_{k=1}^{2p}(-1)^ki_k=0 \mbox{ (mod $n$) }, 1\le i_1\neq i_2\neq \cdots\neq i_{2p}\le n\},\\
B_{2p,s}&=\{(i_1,\ldots,i_{2p})\in \mathbb N^{2p}\suchthat \sum_{k=1}^{2p}(-1)^ki_k=sn,1\le i_1,\ldots,i_{2p}\le n\},\\
B_{2p,s}'&=\{(i_1,\ldots,i_{2p})\in \mathbb N^{2p}\suchthat \sum_{k=1}^{2p}(-1)^ki_k=sn,1\le i_1\neq i_2\neq \cdots\neq i_{2p}\le n\}.
\end{align*}
The following lemma gives the cardinality of $B_{2p,s}$.
\begin{lemma}\label{lem:cardinalityb}
Suppose $|B_{2p,s}|$ denotes   the cardinality of $B_{2p,s}$. Then 
$$
|B_{2p,s}|=\sum_{k=0}^{p+s-1}(-1)^k\binom{2p}{k}\binom{(p+s-k)n+p-1}{2p-1}, \;\;\mbox{ for $s=-(p-1),\ldots, 0,\ldots, p-1$.}
$$ 
\end{lemma} 
\noindent Using this lemma we first prove Theorem \ref{thm:reverseciculant} and then we prove the lemma.

\begin{proof}[Proof of Theorem \ref{thm:reverseciculant}]
Let $e_1,\ldots,e_n$ be the standard unit vectors in $\mathbb R^n$, i.e., $e_i=(0,\ldots,1,\ldots, 0)^t$ ($1$ in $i$-th place).  Therefore we have 
\begin{align*}
(RC_n)e_i=\mbox{$i$-th column}=\sum_{i_1=1}^nX_{i_1}e_{i_1-i+1 \mbox{ mod $n$}},
\end{align*}
for $i=1,\ldots, n$ (we write $e_0=e_n$). Repeating the procedure we get 
\begin{align*}
(RC_n)^2e_i=\sum_{i_1,i_2=1}^nX_{i_1}X_{i_2}e_{i_2-i_1+i \mbox{ mod $n$}},
\end{align*}
for $i=1,\ldots, n$. Therefore in general we get
\begin{align*}
(RC_n)^{2p}e_i&=\sum_{i_1,\ldots,i_{2p}=1}^nX_{i_1}\ldots X_{i_{2p}}e_{i_{2p}-i_{2p-1}\cdots -i_1+i \mbox{ mod $n$}},
\\(RC_n)^{2p+1}e_i&=\sum_{i_1,\ldots,i_{2p+1}=1}^nX_{i_1}\ldots X_{i_{2p+1}}e_{i_{2p+1}-i_{2p}\cdots +i_1-i+1 \mbox{ mod $n$}},
\end{align*}
for $i=1,\ldots, n$.  Therefore the trace of $RC_n^{2p}$ can be written as 
\begin{align*}
\Tr(RC_n^{2p})=\sum_{i=1}^ne_i^t(RC_n)e_i=n\sum_{B_{2p}}X_{i_1}\ldots X_{i_{2p}}.
\end{align*}
Which implies  that 
\begin{align}\label{eqn:revari}
&\lim_{n\to\infty}\frac{\var[\Tr(RC_n^{2p})]}{n^{2p+1}}\nonumber
\\=&\lim_{n\to\infty}\frac{1}{n^{2p-1}}\sum_{B_{2p}}\sum_{B_{2p}}\l(\E[X_{i_1}\ldots X_{i_{2p}}X_{j_1}\ldots X_{j_{2p}}]-\E[X_{i_1}\ldots X_{i_{2p}}]\E[X_{j_1}\ldots X_{j_{2p}}]\r).
\end{align}
Observe that, if $\{i_1,i_2,\ldots,i_{2p}\}\cap \{j_1,j_2,\ldots,j_{2p}\}=\emptyset$ then from independence of $X_i$'s
$$\E[X_{i_1}\ldots X_{i_{2p}}X_{j_1}\ldots X_{j_{2p}}]-\E[X_{i_1}\ldots X_{i_{2p}}]\E[X_{j_1}\ldots X_{j_{2p}}]=0.$$
Therefore we can get the non-zero contribution in \eqref{eqn:revari} only when there is atleast one cross matching among $\{i_1,\ldots,i_{2p}\}$ and $\{j_1,\ldots,j_{2p}\}$, i.e., $\{i_1,i_2,\ldots,i_{2p}\}\cap \{j_1,j_2,\ldots,j_{2p}\}\neq\emptyset$. 

\noindent {\bf Case I (odd number of cross matching):}  Suppose  $(i_1,i_2,\ldots,i_{2p}),(j_1,j_2,\ldots,j_{2p})\in B_{2p}$ and $|\{i_1,i_2,\ldots,i_{2p}\}\cap \{j_1,j_2,\ldots,j_{2p}\}|=2k-1$ for some $k=1,2,\ldots, p$,  where $|\{\cdot\}|$ denotes cardinality of the set  $\{\cdot\}$. We show  that such $(i_1,i_2,\ldots,i_{2p}),(j_1,j_2,\ldots,j_{2p})\in B_{2p}$  will have zero contribution in \eqref{eqn:revari}. 

First observe for $k=1$, either $ \E[X_{i_1}\ldots X_{i_{2p}}]=0$ or $\E[X_{j_1}\ldots X_{j_{2p}}]$ as $\E(X_i)=0$, and $\E[X_{i_1}\ldots X_{i_{2p}}X_{j_1}\ldots X_{j_{2p}}]=0$ as there will be at least one random variable with odd power. Hence for $k=1$,
$$\E[X_{i_1}\ldots X_{i_{2p}}X_{j_1}\ldots X_{j_{2p}}]-\E[X_{i_1}\ldots X_{i_{2p}}]\E[X_{j_1}\ldots X_{j_{2p}}]=0.$$

For $k\geq 2$, let $(\ell_1,\ldots,\ell_{2k-1},i_{2k}, \ldots,i_{2p})$ and $(\ell_1,\ldots,\ell_{2k-1},j_{2k}, \ldots,j_{2p})$ be two typical elements in $B_{2p}$ satisfying the condition of case I. For a  non zero contribution from $\E[X_{i_1}\ldots X_{i_{2p}}X_{j_1}\ldots X_{j_{2p}}]$, there must exist $\ell^*$ and $\ell_*$ among $\{\ell_1,\ldots,\ell_{2k-1}\}$  such that $\ell^*$ and $\ell_*$ matches with one of $\{i_{2k}, \ldots,i_{2p}\}$ and  $\{j_{2k}, \ldots,j_{2p}\}$ receptively, and  the rest of the $(2p-2k)$ variables of $\{i_{2k}, \ldots,i_{2p}\}$ are at least  pair matched and similarly, rest of the $(2p-2k)$ variables of $\{j_{2k}, \ldots,j_{2p}\}$ are also at least pair matched. Hence the number of free variables in $\{i_{2k}, \ldots,i_{2p}\}$ is at most $(p-k)$ and similarly the number of free variables in $\{j_{2k}, \ldots,j_{2p}\}$ is at most $(p-k)$.
After choosing the free variables in $\{i_{2k}, \ldots,i_{2p}\}$, $\ell^*$ will be determined by the rest of the $2k-2$ variables of $\{\ell_1,\ell_2,\ldots,\ell_{2k-1}\}\backslash \{\ell^*\}$ since $(\ell_1,\ldots,\ell_{2k-1},i_{2k},\ldots,i_{2p})\in B_{2p}$. Hence the number of free variables in $\{i_1,i_2,\ldots,i_{2p}\}\cup \{j_1,j_2,\ldots,j_{2p}\}$ satisfying the above condition is at most $(p-k)+(p-k)+(2k-2)=2p-2$. Therefore, for $k\geq 2$
$$
\sum_{B_{2p}}\sum_{B_{2p}}\l(\E[X_{i_1}\ldots X_{i_{2p}}X_{j_1}\ldots X_{j_{2p}}]-\E[X_{i_1}\ldots X_{i_{2p}}]\E[X_{j_1}\ldots X_{j_{2p}}]\r)=O(n^{2p-2}).
$$
Hence the contribution  in \eqref{eqn:revari} is zero if the number of cross matching among $\{i_1,i_2,\ldots,i_{2p}\}$, $\{j_1,j_2,\ldots,j_{2p}\}$ is odd. 

\noindent {\bf Case II (even number of cross matching):} Now suppose  $(i_1,i_2,\ldots,i_{2p})\in B_{2p}$, $(j_1,j_2,\ldots,j_{2p})\in B_{2p}$ and $|\{i_1,i_2,\ldots,i_{2p}\}\cap \{j_1,j_2,\ldots,j_{2p}\}|=2k$ for some $k=1,2,\ldots, p$. We define, for $k=1,2,\ldots,p$,
$$
I_{2k}:=\{((i_1,\ldots,i_{2p}),(j_1,\ldots,j_{2p}))\in B_{2p}\times B_{2p}\suchthat |\{i_1,\ldots,i_{2p}\}\cap \{j_1,\ldots,j_{2p}\}|=2k\}.
$$
Now from the discussion of Case I, we have 
\begin{align}\label{1}
&\lim_{n\to\infty}\frac{\var[\Tr(RC_n^{2p})]}{n^{2p+1}}\nonumber
\\=&\lim_{n\to\infty}\frac{1}{n^{2p-1}}\sum_{k=1}^p\sum_{I_{2k}}\l(\E[X_{i_1}\ldots X_{i_{2p}}X_{j_1}\ldots X_{j_{2p}}]-\E[X_{i_1}\ldots X_{i_{2p}}]\E[X_{j_1}\ldots X_{j_{2p}}]\r).
\end{align}
To understand the contribution from $I_{2k}$, let us consider a typical element of $I_{2k}$, say, $((\ell_1,\ldots,\ell_{2k},i_{2k+1}, \ldots,i_{2p})$,$(\ell_1,\ldots,\ell_{2k},j_{2k+1}, \ldots,j_{2p}))$. 
For such  an element of $I_{2k}$, the number of free variable will be maximum  if  $\{\ell_1,\ldots,\ell_{2k}\}, \{i_{2k+1}, \ldots,i_{2p}\}$, $\{j_{2k+1}, \ldots,j_{2p}\}$ are disjoint sets and $\ell_1,\ldots,\ell_{2k}$ are distinct and the indices of $(i_{2k+1}, \ldots,i_{2p})$ and $(j_{2k+1}, \ldots,j_{2p})$ are odd-even pair matched. 

We say a pair is {\it odd-even pair matched} if    one of the elements of the pair appears at odd position and other one appears at even position. For example, $(1,1,2,2)$ is odd-even pair matched whereas  $(1,2,1,2)$ is only pair matched, not odd-even pair matched.

 Therefore, if  $(i_{2k+1}, \ldots,i_{2p})$ and $(j_{2k+1}, \ldots,j_{2p})$ are odd-even pair matched, then we have
\begin{equation}\label{reduced condition}
\sum_{i=1}^{2k}(-1)^i\ell_i=0\; \;\mbox{mod $n$}.
\end{equation}
Now for $k=2,3,\ldots, p$, the number of vectors $(\ell_1,\ldots,\ell_{2k})$ with distinct elements that satisfy  equation \eqref{reduced condition}  is $O(n^{2k-1})$. The number of free variables in $\{i_{2k+1}, \ldots,i_{2p}\}$ is $(p-k)$ as they are odd-even pair matched and hence we have $O(n^{p-k})$ choices for $\{i_{2k+1}, \ldots,i_{2p}\}$. Similarly, we have $O(n^{p-k})$ choices for \\$\{j_{2k+1}, \ldots,j_{2p}\}$. Hence, for $k=2,3,\ldots, p$, the maximum number of choice for \\$((\ell_1,\ldots,\ell_{2k},i_{2k+1}, \ldots,i_{2p})$,$(\ell_1,\ldots,\ell_{2k},j_{2k+1}, \ldots,j_{2p}))$ is $O(n^{(2k-1)+(p-k)+(p-k)})=O(n^{2p-1})$. In any other situation, like, $\ell_1,\ldots,\ell_{2k}$ are not distinct or one of $(i_{2k+1}, \ldots,i_{2p})$ and $(j_{2k+1}, \ldots,j_{2p})$ are not odd-even pair matched, the number of choices will be $O(n^{2p-2})$. Hence, maximum contribution $(O(n^{2p-1}))$ will come when $\{\ell_1,\ldots,\ell_{2k}\},$ $\{i_{2k+1}, \ldots,i_{2p}\}$, $\{j_{2k+1}, \ldots,j_{2p}\}$ are disjoint sets and $\ell_1,\ldots,\ell_{2k}$ are distinct and the indices of $(i_{2k+1}, \ldots,i_{2p})$ and $(j_{2k+1}, \ldots,j_{2p})$ are odd-even pair matched. As $\{\ell_1,\ell_2,\ldots,\ell_{2k}\}$ are distinct, we have 
$$
\E[X_{\ell_1}\ldots X_{\ell_{2k}}X_{i_{2k+1}}\ldots X_{i_{2p}}]\E[X_{\ell_{1}}\ldots X_{\ell_{2k}}X_{j_{2k+1}}\ldots X_{j_{2p}}]=0.
$$
Therefore, from the above discussion for $k=2,\ldots,p$, we have 
\begin{align}\label{2}
&\lim_{n\to \infty}\frac{1}{n^{2p-1}}\sum_{I_{2k}}\E[X_{i_1}\ldots X_{i_{2p}}X_{j_1}\ldots X_{j_{2p}}]\nonumber\\
&=\lim_{n\to \infty}\frac{1}{n^{2p-1}}\sum_{B_{2k}'}\sum_{B_{2k}'} c_k n^{2(p-k)}\E[X_{i_1}\ldots X_{i_{2k}}X_{j_1}\ldots X_{j_{2k}}]\nonumber \\
&=c_k\lim_{n\to \infty}\frac{1}{n^{2k-1}}\sum_{B_{2k}'}\sum_{B_{2k}'}\E[X_{i_1}\ldots X_{i_{2k}}X_{j_1}\ldots X_{j_{2k}}],
\end{align}
where $c_k=\l(\binom{p}{p-k}^2(p-k)!\r)^2.$ For the odd-even pair matching among $(2p-2k)$ variables from $(i_1,\ldots, i_{2p})$, first we choose $(p-k)$ odd and $(p-k)$ even position from the available $p$ odd and $p$ even position in $\binom{p}{p-k}^2$ ways. After choosing $(p-k)$ odd, $(p-k)$ even positions and $(p-k)$ free variables in odd positions, the random variables in even positions can permute among themselves (satisfying the condition of odd-even pair matching) in  $(p-k)!$  ways. Hence, odd-even pair matching among $(2p-2k)$ variables of $(i_1,\ldots, i_{2p})$  happens in  $\binom{p}{p-k}^2(p-k)!n^{p-k}$ ways. Similarly,  odd-even pair matching happens among $(2p-2k)$ variables from $(j_1,j_2,\ldots, j_{2p})$ in $\binom{p}{p-k}^2(p-k)!n^{p-k}$ ways. 
The rest of the variable $\{i_1,i_2,\ldots,i_{2k}\}$ and $\{j_1,j_2,\ldots,j_{2k}\}$ will (cross) match completely and both belong to $B_{2k}'$. Now from \eqref{2}, we get 
\begin{align}\label{reduced from 2}
&\lim_{n\to \infty}\frac{1}{n^{2p-1}}\sum_{I_{2k}}\E[X_{i_1}\ldots X_{i_{2p}}X_{j_1}\ldots X_{j_{2p}}] \nonumber\\
&=c_k\lim_{n\to \infty}\frac{1}{n^{2k-1}}\sum_{B_{2k}'}\sum_{B_{2k}'}\E[X_{i_1}\ldots X_{i_{2k}}X_{j_1}\ldots X_{j_{2k}}]\nonumber \\
&=c_k\lim_{n\to \infty}\frac{1}{n^{2k-1}}\sum_{s,t=-(k-1)}^{k-1}\sum_{B_{2k,s}'}\sum_{B_{2k,t}'}\E[X_{i_1}\ldots X_{i_{2k}}X_{j_1}\ldots X_{j_{2k}}]\nonumber \\
&=c_k \sum_{s=-(k-1)}^{k-1}\lim_{n\to \infty}\frac{|B_{2k,s}'| (2-{\bf 1}_{\{s=0\}})k!k!}{n^{2k-1}}.
\end{align}
The constant $(2-{\bf 1}_{\{s=0\}})$ appeared because for $s\neq 0$, $B_{2k,s}'$ can (cross) match completely  with $B_{2k,s}'$ and $B_{2k,-s}'$ . 
The factor $k!^2$ appeared because for a fixed choice of $(i_1,\ldots,i_{2k})\in B_{2k,s}'$,  we can choose the same set of values  from $\{j_1,\ldots, j_{2k}\}$ in $k!k!$ different ways, permuting the odd positions and the even positions among themselves. 

Now using Lemma \ref{lem:cardinalityb}, we get
\begin{align}\label{3}
\lim_{n\to \infty}\frac{|B_{2k,s}'|}{n^{2k-1}}=\lim_{n\to \infty}\frac{|B_{2k,s}|}{n^{2k-1}}
&=\frac{1}{(2k-1)!}\sum_{s=-(k-1)}^{k-1}\sum_{j=0}^{k+s-1}(-1)^j\binom{2k}{j}(k+s-j)^{2k-1}.
\end{align}
Note that when  $k=1$ in Case II, then from \eqref{reduced condition} we get  $\ell_1=\ell_2$. Therefore 
\begin{align}\label{4}
&\lim_{n\to \infty}\frac{1}{n^{2p-1}}\sum_{I_{2}}\l(\E[X_{i_1}\ldots X_{i_{2p}}X_{j_1}\ldots X_{j_{2p}}]-\E[X_{i_1}\ldots X_{i_{2p}}]\E[X_{j_1}\ldots X_{j_{2p}}]\r)\nonumber
\\=&c_1(\E X_1^4-(\E X_1^2)^2)=(3-1)c_1=2c_1,
\end{align}
where $c_1=\l(\binom{p}{p-1}^2(p-1)!\r)^2$.
Therefore from \eqref{1},\eqref{reduced from 2},\eqref{3},\eqref{4} we have 
$$
\lim_{n\to\infty}\frac{\var[\Tr(RC_n^{2p})]}{n^{2p+1}}=\sum_{k=2}^pc_kg(k)+2c_1,
$$
 where $c_k=\l(\binom{p}{p-k}^2(p-k)!\r)^2$ and $g(k)$ is given by
 $$
 g(k)=\frac{1}{(2k-1)!}\sum_{s=-(k-1)}^{k-1}\sum_{j=0}^{k+s-1}(-1)^j\binom{2k}{j}(k+s-j)^{2k-1}(2-{\bf 1}_{\{s=0\}})k!k!.
 $$
 Hence the result.
\end{proof}

\begin{proof}[Proof of Lemma \ref{lem:cardinalityb}]
Note that, for fixed positive integers $p$ and $s$,  $|B_{2p,s}|$ is the coefficient of $x^{sn}$ in the expression $(x+x^2+\cdots+x^n)^p(x^{-1}+x^{-2}+\cdots+x^{-n})^p$. We have
\begin{align*}
&(x+x^2+\cdots+x^n)^p(x^{-1}+x^{-2}+\cdots+x^{-n})^p
\\&={x^{p-np}}(1+x+\cdots+x^{n-1})^{2p}
\\&=x^{p-np} (1-x^n)^{2p}(1-x)^{-2p}
\\&=x^{p-np}\l(\sum_{k=0}^{2p}(-1)^k\binom{2p}{k}x^{kn}\r)\l(\sum_{m=0}^{\infty}\binom{2p+m-1}{2p-1}x^m\r).
\end{align*}
It is clear from the last equation that the coefficient of $x^{sn}$ is 
\begin{align*}
\sum_{k=0}^{p+s-1}(-1)^k\binom{2p}{k}\binom{(p+s-k)n+p-1}{2p-1},
\end{align*}
for $s=-(p-1),\ldots, 0,\ldots, p-1$. This completes the proof.
\end{proof}

\vspace{.2cm}
\subsection{Proof of Theorem \ref{thm:scvariance}}
We  use the following notations in the proof of Theorem \ref{thm:scvariance}.
\begin{align*}
B_{p}&=\l\{(j_1,\ldots,j_{p})\suchthat \sum_{i=1}^{p}\epsilon_i j_i=0\; \mbox{(mod n)}, \epsilon_i\in\{+1,-1\}, 1\le j_1,\ldots,j_{p}\le \frac{n}{2}\r\},
\\B_p^{(k)}&=\{(j_1,\ldots,j_p)\in B_p\suchthat j_1+\cdots+j_k - j_{k+1}-\cdots -j_p=0 \;\mbox{ (mod $n$)}\},
\\B_p'^{(k)}&=\{(j_1,\ldots,j_p)\in B_p\suchthat j_1+\cdots+j_k - j_{k+1}-\cdots -j_p=0 \;\mbox{ (mod $n$)}, j_1\neq \cdots \neq j_k \},
\\B_{p,s}^{(k)}&=\{(j_1,\ldots,j_p)\in B_p\suchthat j_1+\cdots+j_k - j_{k+1}-\cdots -j_p=sn\}.
\end{align*}
In  set $B_p$, we collect $(j_1,\ldots,j_{p})$ according to their multiplicity.  The following lemma counts the number of elements of $B_p$.
\begin{lemma}\label{lem:sc}
Suppose $|B_p^{(k)}|$ denotes the number of elements in $B_p^{(k)}$. Then 
\begin{align*}
h_p(k):=
\lim_{n\to \infty}\frac{|B_p^{(k)}|}{n^{p-1}}= \frac{1}{(p-1)!}\sum_{s=-\lceil\frac{p-k}{2}\rceil}^{\lfloor\frac{k}{2}\rfloor}\sum_{q=0}^{2s+p-k}(-1)^q\binom{p}{q}\l(\frac{2s+p-k-q}{2}\r)^{p-1},
\end{align*}
where $\lceil x\rceil$ denotes the smallest integer not less than $x$.
\end{lemma}

\begin{proof}[Proof of Theorem \ref{thm:scvariance}]
 Suppose $n$ is  odd positive integers. We write $ n/2$ instead of $\lfloor n/2\rfloor $, as asymptotic is same as $n\to\infty$. Then
\begin{align*}
\Tr(SC_n^p)&=\sum_{k=0}^{n-1}\lambda_k^p
=\sum_{k=0}^{n-1}\left(X_0+2\sum_{j=1}^{n/2}X_j\cos(\omega_k j)\right)^p
\\&=\sum_{\ell=0}^{p}\binom{p}{\ell}X_0^{p-\ell}\sum_{k=0}^{n-1} \left(\sum_{j=1}^{n/2}X_j(e^{i\omega_k j}+e^{-i\omega_k j})\right)^{\ell},
\end{align*}
where $\omega_k=\frac{2\pi k}{n}$. Since  $\sum_{k=0}^{n-1}e^{i\omega_k j}=0$ for $j\in \mathbb Z\backslash \{0\}$, we have 
\begin{align}\label{eqn:sc0}
\Tr(SC_n^p)=n\sum_{\ell=0}^{p}\binom{p}{\ell}X_0^{p-\ell}\sum_{B_{\ell}}X_{j_1}X_{j_2}\ldots X_{j_{\ell}},
\end{align}
where  $B_{\ell}$ for $\ell=1,\ldots,p$ is given by
\begin{align*}
B_{\ell}:=\l\{(j_1,\ldots,j_{\ell})\suchthat \sum_{i=1}^{\ell}\epsilon_i j_i=0\; \mbox{(mod n)}, \epsilon_i\in\{+1,-1\}, 1\le j_1,\ldots,j_{\ell}\le \frac{n}{2}\r\}
\end{align*} 
and $B_0$ is an empty set with the understanding that the contribution from the sum corresponding to $B_0$ is 1. Note, in $B_{\ell}$, $(j_1,\ldots,j_{\ell})$ are collected according to their multiplicity. 

\noindent\textbf{Proof of (i):} Now assume $p=2m+1$, odd. From \eqref{eqn:sc0}, we have 
\begin{align}\label{eqn:sc1}
\lim_{n\to \infty}\frac{\var[\Tr(SC_n^p)]}{n^{p+1}} 
&=\lim_{n\to \infty}\frac{1}{n^{p-1}}\var\l[\sum_{\ell=0}^{p}\binom{p}{\ell}X_0^{p-\ell}\sum_{B_{\ell}}X_{j_1}X_{j_2}\ldots X_{j_{\ell}}\r]\nonumber
\\&=\lim_{n\to \infty}\frac{1}{n^{p-1}}\E\l[\sum_{\ell=0}^{p}\binom{p}{\ell}X_0^{p-\ell}\sum_{B_{\ell}}X_{j_1}X_{j_2}\ldots X_{j_{\ell}}\r]^2 \nonumber \\
&=\lim_{n\to \infty}\frac{1}{n^{p-1}}\E\l[\sum_{\ell,k=0}^{p}\binom{p}{\ell}\binom{p}{k}X_0^{2p-\ell-k}\sum_{B_{\ell},B_k}X_{j_1}\ldots X_{j_{\ell}}X_{i_1}\ldots X_{i_k}\r].
\end{align}
The second last equality follows from the  fact that for $p$  odd 
\begin{align*}
\E\l[\sum_{\ell=0}^{p}\binom{p}{\ell}X_0^{p-\ell}\sum_{B_{\ell}}X_{j_1}X_{j_2}\ldots X_{j_{\ell}}\r]=0,
\end{align*}
as there exists at least one random variable with odd power. Also note that the right hand side of \eqref{eqn:sc1} has non zero contribution only when  $\ell=k= p-1$ and $\ell=k=p$. Therefore we have 
\begin{align}\label{eqn:sc2}
&\lim_{n\to \infty}\frac{\var[\Tr(SC_n^p)]}{n^{p+1}}\nonumber \\=&\lim_{n\to \infty}\frac{1}{n^{p-1}}\l[p^2\E\sum_{B_{p-1},B_{p-1}}X_{j_1}\ldots X_{j_{p-1}}X_{i_1}\ldots X_{i_{p-1}}+\E\sum_{B_{p},B_{p}}X_{j_1}\ldots X_{j_{p}}X_{i_1}\ldots X_{i_{p}}\r].
\end{align}
Now we calculate the contribution from the first term in \eqref{eqn:sc2}.  Note that $p-1$ is even and  the contribution of the first term will be non zero only when $\sum_{k=1}^{p-1} \epsilon_k=0$,  $\{j_1,\ldots,j_{p-1}\}$ and $\{i_1,\ldots,i_{p-1}\}$ are  disjoint and pair matched with different sign, i.e., one variable of a pair will have positive sign and the other one will have negative sign. Note, positive or negative signs are arising  due to the values of $\epsilon_i$'s.  Therefore we have 
\begin{align}\label{eqn:sc2.5}
\lim_{n\to \infty}\frac{p^2}{n^{p-1}}\E\sum_{B_{p-1},B_{p-1}}X_{j_1}\ldots X_{j_{p-1}}X_{i_1}\ldots X_{i_{p-1}}&=p^2\binom{p-1}{(p-1)/2}^2[\l((p-1)/2\r)!]^2\frac{1}{2^{(p-1)}}.
\end{align}
The factor $\binom{p-1}{(p-1)/2}^2$ appeared because in $\binom{p-1}{(p-1)/2}$ many ways $\sum_{k=1}^{p-1} \epsilon_k=0$ in one $B_{p-1}$ and the summation is over two $B_{p-1}$.  $\l(\frac{p-1}{2}\r)!$ appeared because for each free choice of  $(p-1)/2$  variables among  $\{i_1,\ldots,i_{p-1}\}$ with positive sign, we can choose rest of the $(p-1)/2$  variables with negative sign in $\l(\frac{p-1}{2}\r)!$ ways to have pair matching. Using  same argument for  $\{j_1,\ldots,j_{p-1}\}$, we get another $\l(\frac{p-1}{2}\r)!$ factor. $\frac{1}{2^{(p-1)}}$ arises because $1\leq i_k,j_k\leq n/2$. 

Next we calculate the second term of \eqref{eqn:sc2}. There will be non zero  contribution when there is exactly odd number of cross matching between $\{i_1,\ldots,i_{p}\}$ and $\{j_1,\ldots,j_{p}\}$ and rest of the variables are pair matched with opposite sign. 
Therefore we get
\begin{align*}
\lim_{n\to \infty}\frac{1}{n^{p-1}}\E\sum_{B_{p},B_{p}}X_{j_1}\ldots X_{j_{p}}X_{i_1}\ldots X_{i_{p}}=\lim_{n\to \infty}\frac{1}{n^{2m}}\sum_{k=0}^{m}\sum_{I_{2k+1}}X_{j_1}\ldots X_{j_{2m+1}}X_{i_1}\ldots X_{i_{2m+1}},
\end{align*}
where 
$$
I_{2k+1}=\{((j_1,\ldots,j_{p}),(i_1,\ldots,i_{p}))\in B_{p}\times B_{p}\;:\;|\{j_1,\ldots,j_{p}\}\cap \{i_1,\ldots,i_{p}\}|=2k+1 \}.
$$
Let $((\ell_1,\ldots,\ell_{2k+1},j_{2k+2},\ldots,j_{2m+1}),(\ell_1,\ldots,\ell_{2k+1},i_{2k+2},\ldots,i_{2m+1}))$ be a typical  element in $I_{2k+1}.$ As we  have discussed, non zero contribution will come when $\ell_1,\ldots,\ell_{2k+1}$ are distinct, $\sum_{t=2k+2}^{2m+1}\epsilon_t=0$,  $\{j_{2k+2},\ldots,j_{2m+1}\}$ and $\{i_{2k+2},\ldots,i_{2m+1}\}$ are pair matched with opposite sign. Therefore, like in \eqref{2}, we have 
\begin{align}\label{eqn:sc2.6}
&\lim_{n\to \infty}\frac{1}{n^{2m}}\sum_{k=0}^{m}\sum_{I_{2k+1}}X_{j_1}\ldots X_{j_{2m+1}}X_{i_1}\ldots X_{i_{2m+1}}\nonumber
\\&=\lim_{n\to \infty}\frac{1}{n^{2m}}\sum_{k=0}^{m}a_k(n/2)^{2(m-k)}\sum_{B_{2k+1}',B_{2k+1}'}X_{j_1}\ldots X_{j_{2k+1}}X_{i_1}\ldots X_{i_{2k+1}}\nonumber
\\&=\sum_{k=0}^{m}\frac{a_k}{2^{2(m-k)}}\lim_{n\to \infty}\frac{1}{n^{2k}}\sum_{B_{2k+1}',B_{2k+1}'}X_{j_1}\ldots X_{j_{2k+1}}X_{i_1}\ldots X_{i_{2k+1}},
\end{align}
where $a_k=(\binom{2m+1}{2m-2k}\binom{2m-2k}{m-k} (m-k)!)^2$. $a_k$ factor is coming for pair matching of $(2m-2k)$ many variables in $(i_1,i_2,\ldots,i_{2m+1})$ and  $(j_1,j_2,\ldots,j_{2m+1})$ both  with opposite sign. In $(i_1,i_2,\ldots,i_{2m+1})$, we can choose $(2m-2k)$ variables in $\binom{2m+1}{2m-2k}$ many ways. Out of $(2m-2k)$ variables, $(m-k)$ many variables  can be chosen with positive sign in $\binom{2m-2k}{m-k}$ many ways. After free choice of $(m-k)$ variables with positive sign, rest of the $(m-k)$ variables with negative sign can be chosen in $(m-k)!$ ways. Therefore for pair matching of $(2m-2k)$ many variables  in $(i_1,i_2,\ldots,i_{2m+1})$ with opposite sign, we get $(\binom{2m+1}{2m-2k}\binom{2m-2k}{m-k} (m-k)!)$ factor. Similarly from $(j_1,j_2,\ldots,j_{2m+1})$, we get one more $(\binom{2m+1}{2m-2k}\binom{2m-2k}{m-k} (m-k)!)$ factor.  Also note $B_1'$ is an empty set. Therefore the term corresponding to $k=0$ in \eqref{eqn:sc2.6} is zero.   Hence  from \eqref{eqn:sc2.6}, we get
\begin{align*}
&\lim_{n\to \infty}\frac{1}{n^{2m}}\sum_{k=0}^{m}\sum_{I_{2k+1}}X_{j_1}\ldots X_{j_{2m+1}}X_{i_1}\ldots X_{i_{2m+1}}\\
&=\sum_{k=1}^{m}\frac{a_k}{2^{2(m-k)}}\lim_{n\to \infty}\frac{1}{n^{2k}}\sum_{B_{2k+1}',B_{2k+1}'}X_{j_1}\ldots X_{j_{2k+1}}X_{i_1}\ldots X_{i_{2k+1}}
\\&=\sum_{k=1}^{m}\frac{a_k}{2^{2(m-k)}}\lim_{n\to \infty}\frac{1}{n^{2k}}\sum_{\ell=0}^{2k+1}\binom{2k+1}{\ell}^2 \ell!(2k+1-\ell)! |B_{2k+1}'^{(\ell)}|
\\&= \sum_{k=1}^{m}\frac{a_k}{2^{2(m-k)}}\sum_{\ell=0}^{2k+1}\binom{2k+1}{\ell}^2 \ell!(2k+1-\ell)!\lim_{n\to \infty}\frac{|B_{2k+1}^{(\ell)}|}{n^{2k}}.
\end{align*}
The factor $\binom{2k+1}{\ell}^2$ appeared because in $\binom{2k+1}{\ell}$ ways we can choose $\ell$ many $+1$ from $\{\epsilon_1,\ldots,\epsilon_{2k+1}\}$ in one $B_{2k+1}'$. The factor $(\ell!(2k+1-\ell)!)$ appeared because for each choice of $(i_1,\ldots,i_{2k+1})$ we have  $(\ell!(2k+1-\ell)!)$ many choice for $(j_1,\ldots,j_{2k+1})$. 
Now using Lemma \ref{lem:sc}, we get
\begin{align}\label{eqn:sc3}
&\lim_{n\to \infty}\frac{1}{n^{2m}}\sum_{k=0}^{m}\sum_{I_{2k+1}}X_{j_1}\ldots X_{j_{2m+1}}X_{i_1}\ldots X_{i_{2m+1}}\nonumber
\\&=\sum_{k=1}^{m}\frac{a_k}{2^{2(m-k)}}\sum_{\ell=0}^{2k+1}\binom{2k+1}{\ell}^2 \ell!(2k+1-\ell)! \ h_{2k+1}(\ell).
\end{align}
Hence combining \eqref{eqn:sc2}, \eqref{eqn:sc2.5} and \eqref{eqn:sc3}, for $p=2m+1$ we have
\begin{align*}
&\lim_{n\to \infty}\frac{\var[\Tr(SC_n^p)]}{n^{p+1}}\\
&=(2m+1)^2\binom{2m}{m}^2(m!)^2\frac{1}{2^{2m}}+\sum_{k=1}^{m}\frac{a_k}{2^{2(m-k)}}\sum_{\ell=0}^{2k+1}\binom{2k+1}{\ell}^2 \ell!(2k+1-\ell)! \ h_{2k+1}(\ell),
\end{align*}
where $a_k=(\binom{2m+1}{2m-2k}\binom{2m-2k}{m-k} (m-k)!)^2$ and $p=2m+1$. This completes the proof when $p$ is odd.

\noindent\textbf{Proof of (ii):} Assume  $p=2m$, even. The idea of the proof is similar to the odd case. Here we outline the proof and skip the details. First note that for $p$ even 
\begin{align*}
\E\l[\sum_{\ell=0}^{p}\binom{p}{\ell}X_0^{p-\ell}\sum_{B_{\ell}}X_{j_1}X_{j_2}\ldots X_{j_{\ell}}\r]\neq 0.
\end{align*}
So the analysis goes in the same line as in the proof of Theorem \ref{thm:reverseciculant}. There will non zero contribution only when $\ell=k=p$ in  \eqref{eqn:sc1}.  Like in Theorem \ref{thm:reverseciculant} (see \eqref{4}), $(\E(X_1^4)-\E(X_1^2)^2)$ will contribute here and $b_1$ corresponds to that contribution in \eqref{eqn:sc4}. Finally it can be shown that
\begin{align}\label{eqn:sc4}
\lim_{n\to \infty}\frac{\var[\Tr(SC_n^p)]}{n^{p+1}} =\sum_{k=2}^{m}\frac{b_k}{2^{(m-k)}}\l(\sum_{\ell=0,\ell\neq k}^{2k}\binom{2k}{\ell}^2 \ell!(2k-\ell)! \  h_{2k}(\ell)+\binom{2k}{k}^2g(k)\r)+\frac{b_1}{2^{(2m-1)}},
\end{align}
where $b_k=(\binom{2m}{2m-2k}\binom{2m-2k}{m-k}(m-k)!)^2$, $g(k)$ as in the Theorem \ref{thm:reverseciculant} and $h_{2k}(\ell)$ as in Lemma \ref{lem:sc}. This completes the proof when $p$ is even.
\end{proof}

\noindent Now we   prove Lemma \ref{lem:sc}.

\begin{proof}[Proof of Lemma \ref{lem:sc}]
By expanding $B_p^{(k)}$ we get
\begin{align}\label{eqn:lemsc2}
|B_p^{(k)}|=\sum_{s=-\lceil\frac{p-k}{2}\rceil}^{\lfloor\frac{k}{2}\rfloor}|B_{p,s}^{(k)}|,
\end{align}
where $B_{p,s}^{(k)}=\{(j_1,\ldots,j_p)\in B_p\suchthat j_1+\cdots+j_k - j_{k+1}-\cdots -j_p=sn\}$.  Now following  the same steps as in the proof of Lemma \ref{lem:cardinalityb}, we have 
\begin{align}\label{eqn:lemsc3}
|B_{p,s}^{(k)}|=\sum_{q=0}^{2s+p-k}(-1)^q\binom{p}{q}\binom{(2s+p-k-q)\frac{n}{2}+p-k-1}{p-1}.
\end{align}
Hence the result follows from  \eqref{eqn:lemsc2} and \eqref{eqn:lemsc3}.
\end{proof}

\begin{remark}
(i) To avoid confusion  we have stated all our theorems    with Gaussian input sequence.  But in the proofs of  Theorem \ref{thm:limvarcirculant} and part (i) of Theorem \ref{thm:scvariance},  we mainly use  the following facts about the input sequence:
\begin{equation}\label{minimum assumption}
X_n\mbox{'s are independent, } \E(X_n)=0, \ \E(X_n^2)=1 \ \mbox{for all}\ n\geq 0 \ \mbox{and} \ \sup_{n\geq 0}\E(X_n^{2p})<\infty.\end{equation}
Therefore,  Theorem \ref{thm:limvarcirculant} and part (i) of Theorem \ref{thm:scvariance} hold with the input sequence  $\{X_n : n\geq 0\}$ also, where  $\{X_n : n\geq 0\}$ satisfies condition \eqref{minimum assumption}.\\
(ii) Similarly in the proofs of Theorem \ref{thm:reverseciculant} and part (ii) of Theorem \ref{thm:scvariance}, we use  the fact that $X_n$'s are independent, $\E(X_n)=0$, $\E(X_n^2)=1$ for all $n$, $\E(X_n^4)$ is same for all $n$ and $\sup_{n\geq 0}\E(X_n^{4p})<\infty$. Hence Theorem \ref{thm:reverseciculant} and part (ii) of Theorem \ref{thm:scvariance} are true if the input sequence $\{X_n:n\geq 0\}$ satisfies these conditions.\\
(iii) We have calculated the limiting formula of the variance $\Tr(A_n^p)$ for a fixed positive integer $p$, but following the idea of our proof one can calculate the limiting formula of the variance $\Tr(f(A_n))$ where $f$ is a fixed polynomial and $A_n$ is one of the patterned matrix of dimension $n\times n$. 
\end{remark}

\vspace{.2cm}
\subsection{Concluding remarks}
In Theorems \ref{thm:main:cir} and \ref{thm:main}, we proved CLT type results in total variation norm of linear spectral statistics of some patterned matrices with Gaussian input sequence.  
It would be interesting to establish CLT type results in total variation norm  of linear spectral statistics of these patterned matrices for any  i.i.d. input sequence with sufficient moment assumption. In \cite{liu2009limit}, the authors have established the CLT for linear spectral statistics of band Toeplitz and Hankel matrices.  It would be interesting  to consider the CLT problems for  appropriate  band and sparse version of circulant, symmetric circulant and reverse circulant matrices. We are currently working on the above issues.

\section{Appendix}\label{sec:norm}
In this section  we present a proof of Theorem \ref{thm:normcir}. We mainly follow the idea of Meckes in \cite{meckes}.   The following lemma is the key ingredient  in the proof of Theorem \ref{thm:normcir}. The lemma gives an upper bound on the expected spectral norms of the random circulant, symmetric circulant, reverse circulant and Hankel  matrices.
 
\begin{lemma}\label{lem:cirnorm}
Suppose $A_n$ is one of the matrix among the random circulant, symmetric circulant, reverse circulant and Hankel matrices with Gaussian input sequence. Then 
\begin{align*}
\E{\|A_n\|}\le c_1 \sqrt{n\log n},
\end{align*}
where $c_1$ is a positive constant.
\end{lemma}

Next we state a result with out proof which will be used in the proof of Theorem \ref{thm:normcir}. The result says that the random Lipschitz function on $\mathbb R^n$ with Lipschitz constant bounded above by $1$ is concentrated  around its mean.
\begin{result}\label{res:logsoblev}
Suppose $F$ is Lipschitz function on $\mathbb R^n$ with Lipschitz constant bounded above by $1$, then for all $t>0$
\begin{align*}
\P[F(X_1,\ldots,X_n)\ge \E[F(X_1,\ldots,X_n)]+t]\le e^{-\frac{t^2}{2}},
\end{align*}
 where $X_1,\ldots, X_n$ are i.i.d. standard normal random variables.
\end{result}

We refer the reader to see Section 2.1-2.3 in \cite{ledoux} for the proof of Result \ref{res:logsoblev} and concentration related  results.  We proceed to prove Theorem \ref{thm:normcir} assuming Lemma \ref{lem:cirnorm}.

\begin{proof}[Proof of Theorem \ref{thm:normcir}]
The spectral norm of a matrix is less than the Hilbert-Schmith norm of that matrix, therefore
\begin{align*}
\|C_n\|\le \sqrt{n\sum_{i=0}^{n-1}X_i^2}.
\end{align*}
This implies that the map $(X_0,X_1,\ldots,X_{n-1})\to \|C_n\|$ is Lipschitz with Lipschitz constant bounded by $\sqrt n$. By Result \ref{res:logsoblev}, we have 
\begin{align*}
\P[\|C_n\|\ge \E\|C_n\|+t]\le e^{-\frac{t^2}{2n}}\;\; \mbox{ for all $t>0$}.
\end{align*}
By Lemma \ref{lem:cirnorm}, we have 
\begin{align*}
\P[\|C_n\|\ge (c_1+\sqrt{2c})\sqrt{n\log n}]\le \frac{1}{n^2}.
\end{align*}
Using Borel-Cantelli lemma, we get the result for random circulant matrix. Proof for symmetric circulant matrix is similar to the proof for circulant matrix. Here we skip the details.

The result holds for random reverse circulant matrix by the fact that $\|C_n\|=\|RC_n\|$, as the absolute values of the eigenvalues of reverse circulant and circulant matrices are same (see \cite{bhs:lsd}, Section 2.2), and circulant and reverse circulant matrices are normal matrices.

Observe that $n\times n$ Hankel matrix $H_n$ can be considered as a $n\times n$ first principle block matrix of $2n\times 2n$ reverse circulant matrix $RC_{2n}$. Therefore we have 
$$
\|H_n\|\le \|RC_{2n}\|.
$$
Hence the result holds for random Hankel matrix.
\end{proof}

It remains to prove Lemma \ref{lem:cirnorm}. The key ingredient for the proof of Lemma \ref{lem:cirnorm} is Dudley's entropy bound \cite{dudley} for the supremum of a subgaussian random process.
Suppose $\{Y_x\suchthat x\in [0,1]\}$ is a random process. Let $d$  be a pseudo-metric on $[0,1]$, defined as 
$$
d(x,y)=\sqrt{\E|Y_x-Y_y|^2}.
$$
The process $\{Y_x:x\in M\}$ is called {\it subgaussian} if
\begin{align}\label{eqn:subgaussian}
\P[|Y_x-Y_y|\ge t]\le 2\exp\left[-\frac{bt^2}{d(x,y)^2}\right],\;\;\mbox{for all $x,y\in [0,1]$ and $t>0$},
\end{align}
for some constant $b>0$. For $\epsilon>0$, the $\epsilon$-covering number of $([0,1],d)$, $N(M,d,\epsilon)$, is the smallest cardinality of a subset $\mathcal N\subset [0,1]$ such that for every $x\in [0,1]$, there exists a $y\in \mathcal N$ such that $d(x,y)\le \epsilon$. The Dudley's entropy bound is following:
\begin{result}\label{res:dudley}
Let $\{Y_x: x \in M\}$ be a subgaussian random process with $\E Y_x=0 $ for every $x\in M$. Then 
$$
\E\sup_{x\in M}|Y_x|\le K\int_{0}^{\infty}\sqrt{\log N(M,d,\epsilon)}d\epsilon,
$$
where $K>0$ depends only on the constant $b$ (as in \eqref{eqn:subgaussian}).
\end{result}

We refer the reader to see Proposition 2.1 in \cite{talagrand} for the above version of the statement and proof.
\begin{proof}[Proof of Lemma \ref{lem:cirnorm}]
\textbf{Circulant matrix: }Recall, the eigenvalues  $\lambda_k$ of circulant matrix $C_n$ are given by
$$
\lambda_k=\sum_{j=0}^{n-1}X_je^{i\omega_kj}=\sum_{j=0}^{n-1}X_j\cos(\omega_k j)+i\sum_{j=0}^{n-1}X_j\sin(\omega_k j),
$$
where $\omega_k=\frac{2\pi k}{n}$ and  $k=1,2,\ldots, n$. Observe that a circulant matrix is a normal matrix, i.e., $C_nC_n^*=C_n^*C_n$. Therefore the spectral norm of $C_n$ is given by 
$$
\|C_n\|=\sup\{|\lambda_k|\suchthat k=1,\ldots,n\}.
$$
Consider a random process $\{Y_x=Y_{x}^{(1)}+iY_x^{(2)}\suchthat x\in [0,1]\}$, where  
$$
Y_x^{(1)}=\sum_{j=0}^{n-1}X_j\cos(2\pi x j)\;\;\mbox{ and }\;Y_x^{(2)}=\sum_{j=0}^{n-1}X_j\sin(2\pi x j)\;\;\mbox{for $0\le x\le  1$},
$$
and $\{X_n:n\geq 0\}$ is a sequence of standard normal random variables. Therefore using triangle inequality,   $\|C_n\|$ is bound by
$$
\|C_n\|\le \sup\{|Y_x^{(1)}|\suchthat {x\in [0,1]}\}+\sup\{|Y_x^{(2)}|\suchthat {x\in [0,1]}\}.
$$
Note that, for $0\le x,y\le 1$, we have 
\begin{align*}
\E|Y_x^{(1)}-Y_y^{(1)}|^2&=\sum_{j=0}^{n-1}(\cos(2\pi x j)-\cos(2\pi y j))^2,\\
\E|Y_x^{(2)}-Y_y^{(2)}|^2&=\sum_{j=0}^{n-1}(\sin(2\pi x j)-\sin(2\pi y j))^2.
\end{align*}
Note that \eqref{eqn:subgaussian} holds for the processes $\{Y_x^{(1)}:x\in[0,1]\}$ and $\{Y_x^{(2)}:x\in[0,1]\}$. Since $|\cos s-\cos t|\le |s-t|$ and $|\sin s-\sin t|\le |s-t|$,  we have 
\begin{align*}
d^{(k)}(x,y)=\sqrt{\E[|Y_x^{(k)}-Y_y^{(k)}|^2]}\le 2\pi |x-y| \sqrt{\sum_{j=0}^{n-1}j^2}\le 4 n^{\frac{3}{2}}|x-y|,\;\mbox{ for $k=1,2$.}
\end{align*}
 Therefore the $\epsilon$-covering number of $([0,1],d^{(k)})$ is
\begin{align*}
N([0,1],d^{(k)},\epsilon)\le N\l([0,1],|\cdot|,\frac{\epsilon}{4 n^{\frac{3}{2}}}\r)\le \frac{4 n^{\frac{3}{2}}}{\epsilon}, \;\mbox{ for $k=1,2$.}
\end{align*}
But, always $d^{(k)}(x,y)\le 2\sqrt{n}$ for $k=1,2$, as $|\cos t|\le 1$ and $|\sin t|\le 1$, and hence $N([0,1],d^{(k)}, \epsilon)=1$ if $\epsilon>2\sqrt n$ for $k=1,2$. Therefore by Result \ref{res:dudley}, we have 
\begin{align*}
\E\|C_n\|\le 2K \int_{0}^{2\sqrt{n}}\sqrt{\log\l(\frac{4 n^{\frac{3}{2}}}{\epsilon}\r)}d\epsilon.
\end{align*} 
By substituting $\epsilon=4 n^{\frac{3}{2}}e^{-t}$, and by integration by parts, we have
\begin{align*}
\E\|C_n\|\le 8K n^{\frac{3}{2}}\int_{\log n}^{\infty}\sqrt t e^{-t}dt\le 8Kn^{\frac{3}{2}}\left(\frac{\sqrt{\log n}}{n}+\frac{1}{2n\sqrt{\log n}}\r). 
\end{align*}
Hence the result holds for random circulant matrix. The result holds for reverse circulant and Hankel matrices from the fact that $\|C_n\|=\|RC_n\|$ and $\|H_n\|\le \|RC_{2n}\|$, as  explained in the proof of Theorem \ref{thm:normcir}.

\vspace{.2cm}
\noindent \textbf{Symmetric circulant matrix: } The eigenvalues  $\{\lambda_k\suchthat k=0,\ldots,n-1\}$ of an $n\times n$ symmetric circulant matrix (see \cite{bhs:lsd}, Section 2.2) are given by,
for $n$ odd,
\begin{align*}
\lambda_k=x_0+2\sum_{j=1}^{\lfloor\frac{n}{2}\rfloor}x_j\cos(\omega_k j)\;\;\mbox{for $k=0,1,\ldots,\lfloor\frac{n}{2}\rfloor$},
\end{align*} 
and $\lambda_{n-k}=\lambda_k$ for $k=1,2,\ldots,\lfloor\frac{n}{2}\rfloor$,
and for $n$ even,
\begin{align*}
\lambda_k=x_0+2\sum_{j=1}^{\frac{n}{2}-1}x_j\cos(\omega_k j)+(-1)^kx_{\frac{n}{2}}\;\;\mbox{for $k=0,1,\ldots,\frac{n}{2}$},
\end{align*} 
and $\lambda_{n-k}=\lambda_k$ for $k=1,2,\ldots,\frac{n}{2}$.

Also recall that our symmetric circulant matrix is constructed using Gaussian input sequence $\{X_n:n\geq 0\}$.
Since the symmetric circulant matrix is a symmetric matrix, we have
\begin{align*}
\E\|SC_n\|&=\E\sup\l\{|\lambda_k|\suchthat k=0,1,\ldots,n-1\r\}
\\&\le 2\E|X_0|+2\E\sup\l\{\l|\sum_{j=1}^{\lfloor\frac{n}{2}\rfloor}X_j\cos(\omega_k j)\r|\suchthat k=1,\ldots,\lfloor\frac{n}{2}\rfloor\r\},
\end{align*} 
for   both odd and even $n$. As $\E|X_0|=\sqrt\frac{2}{\pi}$, we have 
$$
\E\|SC_n\|\le 2\sqrt\frac{2}{\pi}+\E\sup\{|Y_x|\suchthat 0\le x\le 1\}, \mbox{ where $Y_x=\sum_{j=1}^{\lfloor\frac{n}{2}\rfloor}X_j\cos(2\pi x j)$}.
$$
Now following  the same argument given in the proof for circulant matrix case,
we have 
$$
\E\|SC_n\|\le C.\sqrt{n\log n},
$$
for some positive constant $C$. This completes the proof of the lemma.
\end{proof}

\begin{remark}
In Theorem \ref{thm:normcir}, we show that  $\limsup_{n \to \infty}\frac{\|A_n\|}{\sqrt{n\log n}}\le C \;\mbox{ a.s.},
$ where $A_n$ is one of the $n\times n$ matrix among random circulant, symmetric circulant, reverse circular and Hankel matrices with Gaussian input sequence. But this result holds even if the elements of the input sequence  are  independent, symmetric and uniformly subgaussian, as Result \ref{res:logsoblev}  and Lemma \ref{lem:cirnorm}   are true for  independent, symmetric and uniformly subgaussian entries  also (see  \cite{meckes}).
\end{remark}

\vspace{.2cm}
\noindent{\bf Acknowledgement:} The authors would like to thank Niranjan Balachandran and Manjunath Krishnapur for useful comments. 

\bibliography{bibtex}
\bibliographystyle{amsplain}
\end{document}